\documentclass[11pt]{amsart}
\usepackage[english]{babel}
\usepackage{multicol}
\usepackage[centertags]{amsmath}
\usepackage{amsfonts}
\usepackage{amssymb,latexsym,url,amsthm}
\usepackage{enumerate}
\usepackage{version}

\usepackage{geometry}
\newtheorem{thm}{Theorem}[section]

\newtheorem{lem}[thm]{Lemma}
\newtheorem{prop}[thm]{Proposition}

\theoremstyle{definition}
\newtheorem{defn}[thm]{Definition}
\theoremstyle{remark}
\newtheorem{rem}[thm]{Remark}
\numberwithin{equation}{section}

\def \<{\langle}
\def \>{\rangle}

\newcommand{\R}{\mathbb{R}}
\newcommand{\N}{\mathbb{N}}

\begin{document}

\title{A Maximum Principle for the controlled Sweeping Process}

\author[Chems Eddine Arroud]{Chems Eddine Arroud}
\address[Chems Eddine Arroud]{Department of Mathematics, Jijel University, Jijel, Algeria, and Mila University Center, Mila, Algeria}
\email{arroud.math@gmail.com}

\author[Giovanni Colombo]{Giovanni Colombo}
\address[Giovanni Colombo]{Universit\`a di Padova, Dipartimento di Matematica and I.N.d.A.M research unit, via Trieste 63, 35121 Padova, Italy}
\email{colombo@math.unipd.it}

\thanks{This work was done while the first author was visiting the Department of Mathematics of Padova University, funded by 
\textit{Programme Boursier ``PNE''
du Minist\`ere de l'Einsegnement Sup\'erieur et de la Recherche Scientifique, R\'epublique Alg\'erienne}. The second author is partially supported by
Padova University Research Project PRAT 2015 ``Control of dynamics with active constraints''.}



\date{\today}
\begin{abstract}
We consider the free endpoint Mayer problem for a controlled Moreau process, the control acting as a perturbation of the dynamics driven by the
normal cone, and derive
necessary optimality conditions of Pontryagin's Maximum Principle type. The results are also discussed through an example.
We combine techniques from \cite{ML} and from \cite{BPK}, which in particular deals with a different but related control problem.
Our assumptions include the smoothness of the boundary of the moving set $C(t)$, but, differently from \cite{BPK}, 
do not require strict convexity.
Rather, a kind of inward/outward pointing condition is assumed on the reference optimal trajectory at the times where the boundary of $C(t)$ is touched.
The state space is finite dimensional.
\end{abstract}
\maketitle
\section{Introduction}\label{intro}
Moreau's sweeping process appears as a model in several contexts and is being studied from the theoretical viewpoint since the early '70s
of last Century. The main subject of investigation continues to be the existence of solutions, under increasing degrees of generality. 

Essentially, the sweeping process is an evolution differential inclusion, which models the displacement of a point subject to be 
dragged by a moving set in a direction normal to its boundary. Formally, the (perturbed) sweeping process is the differential inclusion
\begin{equation}\label{SPintro}
\dot{x}(t)\in - N_{C(t)} (x(t)) + f(x(t)) ,\quad t\in [0,T]
\end{equation}
coupled with the initial condition
\begin{equation}\label{ICintro}
x(0)=x_0\in C(0), 
\end{equation}
where $C(t)$ is a closed moving set, with normal cone $N_{C(t)}(x)$ at $x\in C(t)$,
and the space variable belongs to a Hilbert space (to $\R^n$ in the present paper). If $C(t)$ is convex,
or mildly non-convex (i.e., uniformly prox-regular), and is Lipschitz as a set-valued map depending on the time $t$, and the perturbation
$f$ is Lipschitz, then it is well known that the Cauchy problem \eqref{SPintro}, 
\eqref{ICintro} admits one and only one Lipschitz solution (see, e.g., \cite{thib}).
Observe that the state constraint $x(t)\in C(t)$ for all $t\in [0,T]$ is built in the dynamics, being $N_{C(t)}(x)$
empty if $x\not\in C(t)$.

The present paper deals with the problem of determining necessary conditions for global minimizers of a final cost $h(x(T))$,
subject to the finite dimensional controlled sweeping dynamics
\begin{equation}\label{intro:CP}
\dot{x}(t)\in - N_{C(t)} (x(t)) + f(x(t), u(t)),\quad x(0)=x_0\in C(0),\quad u(t)\in U ,\quad t\in [0,T],
\end{equation}
$U$ being the control set and $f$ being smooth. 
Given a global minimizer, we prove that for a suitable adjoint vector, which is a BV function that satisfies a natural
ODE in the sense of distributions together with the usual final time transversality condition,
a version of Pontryagin's Maximum Principle holds. To keep unnecessary technicalities at a minimum, we do not add (further) state
constraints to the final point. Our main assumptions are smoothness of the 
boundary of the moving set $C(t)$ and, more importantly, a kind of outward/inward pointing condition on 
$f(x_\ast(t),u_\ast(t))$ at times $t$ where the optimal trajectory $x_\ast(t)$ belongs to the boundary of $C(t)$ (see \eqref{m1}
or \eqref{m2} below). This strong requirement is assumed in order to handle the discontinuity of the gradient of the distance $d_{C(t)}(\cdot)$
to the set $C(t)$ at boundary points. In fact, the main difficulty to be overcome in the study of necessary conditions for optimal control problems
subject to \eqref{intro:CP} is the severe lack of Lipschitz continuity of the normal cone mapping at boundary points of $C(t)$. The outward/inward
pointing condition on $f$ indeed permits to confine this issue at a negligible time set.

Control problems driven by a dynamics which involves the sweeping process appeared rather recently. Not mentioning some works scattered
in the mechanical engineering literature, some early theoretical results 
appeared in \cite{serea1} 
(a Hamilton-Jacobi characterization of the value function, with $C$ constant, later generalized in \cite{GCMP}) and in \cite{rindl1,rindl2} 
(existence and discrete approximation of optimal controls,
in the related framework of rate independent processes). More recently, the papers \cite{CHHM,CHHM1,CHHM2} are devoted to
the case where the control acts on the moving set,
which in turn is required to have a polyhedral structure. In particular, \cite{CHHM2} contains a set of necessary conditions for local minima 
which are derived by 
passing to the limit along suitable discrete approximations. Some partial results on necessary conditions for an optimal control
problem acting on the perturbation $f$ were obtained in \cite{SSER}, while the first complete achievement of this type appeared in \cite{BPK}.
The present paper owes to \cite{BPK} several ideas. The problem studied in \cite{BPK} involves a controlled ODE, coupled with a sweeping process with
a constant moving set $C$, and an adjoint equation together with Pontryagin's Maximum Principle are derived by passing to the limit
along suitable Moreau-Yosida approximations. The set $C$ is required to be both smooth and uniformly convex. The dynamics considered
in \cite{BPK} is different from \eqref{intro:CP}, but the main difficulty -- namely the discontinuity of $\nabla d_C(\cdot)$ at boundary points --
is exactly the same. In \cite{BPK}, this issue is solved by imposing enough smoothness on $\partial C(t)$ and via a smooth extension
of $d_C$ up to the interior of $C(t)$. This method, however, requires the uniform convexity of $C(t)$, which in turn yields the coercivity
of the Hessian of the (modified) squared distance. This last property is important to obtain a uniform $L^1$ bound on a sequence
of approximate adjoint vectors, which provides compactness in the space of $BV$ functions of the time variable. 
Our contribution is in modifying the method developed
in \cite{BPK} -- through seemingly simpler estimates on the distance $d_C$ based on \cite{ML} -- in order to drop the requirement of uniform convexity. The
price to pay is the inward/outward pointing assumption \eqref{m1} or \eqref{m2}. A simple example permits to test our necessary
conditions.

The recent results
contained in \cite{caom1,caom2} are also worth being mentioned. In such papers the control acts both on the moving set, which is required
to be polyhedral, and on the perturbation $f$. The problem studied is on one hand more general,
on the other the method requires some extra regularity assumptions on the optimal trajectory. Also, in contrast with our approach,
the moving set is allowed to be nonsmooth, but its generality is weakened by the requirement to be a polyhedron.
This happens because of the need of computing explicitly the coderivative of the normal cone mapping.
Indeed, the method used in \cite{caom1,caom2} is \textit{completely different} from the one adopted in the present paper,
as it relies on passing to the limit along a suitable sequence of discrete approximations of the reference optimal trajectory. The necessary
conditions obtained in \cite{caom2} include a kind of adjoint equation, transversality conditions both at the initial and at the final point,
and nontriviality conditions, but do not include a maximum principle. General existence and relaxation results for optimal control problems
of the same nature of those investigated in \cite{caom1,caom2} appear in \cite{tol}.

Finally, let us mention that H. Sussmann devoted a lot of work to establish the Maximum Principle in high generality, including 
possibly discontinuous vector fields (see, e.g., \cite{suss}). Here we rely on the special structure of the right hand side of \eqref{intro:CP}, 
and develop an \textit{ad hoc} method.

In what follows, Section \ref{stat} contains the statement of overall assumptions and of the main result, 
while Sections \ref{Dynamic} to \ref{convergence} are devoted to the proofs. In Section \ref{Ex} an example is presented and discussed.

\section{Preliminaries}\label{prelim}
We will consider all vectors in a finite dimensional space as column vectors.

Let $C\subset \R^n$ be nonempty and closed. We denote by $d_C(x)$ the distance of $x$ from $C$, $d_C(x):=\inf \{
|y-x|:y\in C\}$, and the metric projection of $x$ onto $C$ is the set of points in $C$ which realize the infimum.
Should this set be a singleton, we denote this point by ${\rm proj}_C(x)$. Given $\rho >0$, we set 
\[
C_\rho := \{ x : d_C(x) < \rho\}.
\]
\textit{Prox-regular sets} will play an important role in the sequel. The definition was first given by Federer, 
under the name of \textit{sets with positive reach}, and later studied by several authors (see the survey paper \cite{CT}).
We give only the definition for \textit{smooth} sets, because the general case will not be relevant here. 
All definitions of tangent and normal cones may be found in \cite{CLSW}, to which we refer
for all concepts of nonsmooth analysis that will be touched within this paper.
By a smooth set we mean
a closed set $C$ in $\R^n$ whose boundary is an embedded manifold of dimension $n-1$. In this case, the tangent cone is actually
an $n-1$-dimensional vector space
and the normal cone is a half ray (or a line if $C$ has empty interior). In particular, we will consider sets whose
boundary can be described as the zero set of a smooth function (at least of class $\mathcal{C}^{1,1}$, namely of class $\mathcal{C}^1$
with Lipschitz partial derivatives), with nonvanishing gradient.
\begin{defn} \label{posreach} 
Let $C\subset \R^n$ be a closed smooth set and $\rho >0$ be given. We say that $C$ is $\rho$-prox-regular provided the inequality
\begin{equation}\label{ineqphi}
\langle \zeta,y-x\rangle\le \frac{| y-x|^2}{2\rho}
\end{equation}
holds for all $x,y\in C$, where $\zeta$ is the unit external normal to $C$ at $x\in \partial C$.
\end{defn}
In particular, every convex set is $\rho$-prox regular for every $\rho >0$ and every set with a $C^{1,1}$-boundary is $\rho$-prox regular,
where $\rho$ depends only the Lipschitz constant of the gradient of the parametrization of the boundary (see \cite[Example 64]{CT}).
In this case, the (proximal) normal cone to $C$ at $x\in C$ is the nonnegative half ray generated by the unit external normal, and
\[
v\in  N_C(x)\text{ if and only if there exists $\sigma >0$ such that }\langle v,y-x\rangle \le \sigma |y-x|^2\;\,\,\forall y\in C.
\]
Prox-regular sets enjoy several properties, including uniqueness of the metric projection and differentiability of the 
distance (in a suitable neighborhood) and normal regularity, which hold also true for convex sets, see, e.g. \cite{CT}.
We state the main properties which we are going to use in the present paper.
\begin{prop}\label{propdist}
Let $\rho >0$ be given and let $C\subset \R^n$ be $\rho$-prox-regular. Then $d_C$ is differentiable on $C_{\rho}\setminus C$, and 
\[
\nabla d_C(x)=(x-{\rm proj}_C(x))/d_C(x)\;\text{ for all }x\in C_\rho\setminus C.
\]
Moreover, $\nabla d_C$ is Lipschitz with Lipschitz constant $2$ in $C_{\frac{\rho}{2}}\setminus C$. 
Finally, ${\rm proj}_C$ is well defined and is Lipschitz with Lipschitz constant $2$ in $C_{\frac{\rho}{2}}$.
\end{prop}
The proof of this Proposition can be found, e.g., in \cite{CT}.

The fact that the distance from a $\rho$-prox-regular set $C$ is of class ${\mathcal C}^{1,1}$ in $C_{\frac{\rho}{2}}\setminus C$
will play a fundamental role in the sequel. In particular, given a moving closed set $C(t)$, $t\in [0,T]$, 
we wish to discuss the differentiability of the Lipschitz map 
\begin{equation}\label{defP}
x\mapsto x-\text{proj}_{C(t)}(x) := P(t,x),
\end{equation}
namely of the gradient $\nabla_x$ of $\frac12 d^2_{C(t)}(x)$ with respect to the state variable $x$, 
in the case where the boundary of $C(t)$ is an immersed
manifold of class ${\mathcal C}^2$. Then it is well known that $d^2_{C(t)}(\cdot)$ is of class ${\mathcal C}^2$ in $C_{\rho}(t)\setminus C(t)$,
where $\rho$ depends only on the global lower bound of the curvature of $\partial C(t)$, see, e.g., \cite[Theorem 3.1]{AS}.
If $x\in {\rm int}\,C(t)$, then $P(t,x)$ vanishes in a neighborhood of $x$ and so it is differentiable in the classical sense, with zero Jacobian. 
If $x \notin C(t)$, then $P(t,x)=d_{C(t)}(x)\nabla_{x}d_{C(t)}(x)$, so that
\begin{equation}\label{D2d}
\nabla_{x}P(t,x)= d_{C(t)}(x)\nabla_{x}^{2}d_{C(t)}(x)+ \nabla_{x}d_{C(t)}(x)\otimes \nabla_{x}d_{C(t)}(x)=: P_x(t,x),
\end{equation}
where we recall that if $v$ and $w$ are column vectors, then $v\otimes w$ denotes the matrix $v \, w^\top$ and we denote by $\nabla_x^2$
the Hessian with respect to the state variable $x$.

Denote by $n(t,x)$ the unit external normal to $C(t)$ at $x$, if $x\in \partial C(t)$, and $0$ if $x \in {\rm int}\, C(t)$.
We will adopt the following convention:
\begin{center}
if $x\in \partial C(t)$, by writing $\nabla_{x}d_{C(t)}(x)$ we mean $n(t,x)$.
\end{center}
With this notation, one can extend       
$\nabla_{x}P(t,x)$ also to $x\in \partial C(t)$, by using \eqref{D2d}. Of course, this does not mean that $P$ is differentiable at $\partial C(t)$: 
the Clarke generalized gradient (with respect to $x$) of $P(t,x)$ at $x\in \partial C(t)$ is $\partial_x P(t,x) = {\rm co}\, \{ 0 , P_x(t,x)\}$.

We will consider also the signed distance
\[
d_S(t,x):= 
\begin{cases}
d_{C(t)}(x)&\text{ if }x \in C^c(t)\\
-d_{C(t)^c}(x)&\text{ if }x\in C(t),
\end{cases}
\]
where $C(t)^c$ denotes the complement of $C(t)$. It is well known that $d_S(t,\cdot)$ is of class ${\mathcal C}^{2}$ around $\partial C(t)$
(see, e.g., Proposition 2.2.2 (iii) in \cite{CS}) if $\partial C (t)$ is a manifold of class ${\mathcal C}^2$ and $C(t)$ has nonempty interior.
It is also easy to see that $d_S(\cdot,x)$ is Lipschitz if $C(\cdot)$ is so.
\section{Standing assumptions and statement of the main results}\label{stat}
The following assumptions will be valid throughout the paper.
\begin{itemize}
\item[$(H_{1})$:]\quad $C:[0,\infty)\leadsto \R^{n}$ is a set-valued map with the following properties:
	\begin{itemize}
	\item[$(H_{1.1})$:]\quad for all $t \in [0,T]$, $C(t)$ is nonempty and compact and there exists $\rho >0$ such that $C(t)$ is 
$\rho$-prox regular. Moreover, $C(t)$ has a $C^{3}$-boundary.
	\item[$(H_{1.2})$:]\quad $C$ is $\gamma$-Lipschitz.
	\end{itemize}
\item[$(H_{2})$:]\quad $U\subset \R^{m}$ is compact and convex.
\item[$(H_{3})$:]\quad $f: \R^{n}\times U \rightarrow \R^{n}$ is a single valued map with the following properties:
	\begin{itemize}
	\item[$(H_{3.1})$:]\quad that there exist $\beta\geq 0 $ such that $\vert f(x,u)\vert \leq \beta$ for all $(x,u)$;
	\item[$(H_{3.2})$:]\quad $f(\cdot,\cdot)$ is of class $C^{1}$;
	\item[$(H_{3.3})$:]\quad $f(\cdot,\cdot)$ is Lipschitz with Lipschitz constant $k$;
	\item[$(H_{3.4})$:]\quad $f(x,U)$ is convex for all $x\in\R^{n}$;
	\end{itemize}
\item[$(H_{4})$:]\quad $h: \R^n \rightarrow \R $ is of class $C^{1}$.	
\end{itemize} 
If $C(t)=\lbrace x: g(t,x)\leq 0 \rbrace$, with $g(\cdot,x)$ Lipschitz, $g(t,\cdot)$ 
of class $C^{2,1}$ is possible to impose conditions directly on the map $g$ in order to let $(H_{1.1})$ and $(H_{1.2})$ hold.
This is discussed in detail in \cite{ANT}.

We are interested in determining necessary conditions for solutions of the following minimization problem, that we call \textit{Problem $(P)$}:

\medskip

Minimize $h(x(T))$ subject to 
\begin{equation}  \label{probmin}
\left\{ \begin{array}{l}
\dot{x}(t)\,\in \,  -N_{C(t)}(x(t)) + f(x(t),u(t))),
\\
x(0)=x_{0}\in C(0)\;,
\\
\end{array}\right.
\end{equation}
with respect to $u:[0,T] \rightarrow U$, $u$ measurable (labeled as \textit{admissible control}).

\medskip

Let us recall that the dynamics \eqref{probmin} implicitly contains the state constraint
\[
x(t)\in C(t)\qquad\forall t\in [0,T]. 
\]
Existence of minimizers for $(P)$ can be obtained by standard methods (even under less stringent assumptions on $f$), essentially thanks to the
graph closedness of the normal cone to a prox-regular set (see, e.g., \cite[Proposition 7]{CT}).

Let $(x_{\ast},u_{\ast})$ be a minimizer. We will impose an \textit{outward (resp. inward) pointing condition} on 
$f(x_\ast(t),u_\ast(t))$ with respect to 
the boundary of $C(t)$. To this aim, we introduce the (possibly empty) set
\begin{equation}\label{defidelta}
I_{\partial}:= \lbrace t\in [0,T]:x_{\ast}(t)\in  \partial C(t)\rbrace
\end{equation}
and require that there exists $\sigma >0$ 
for which either
\begin{equation}\label{m1}\tag{$M_1$}
\frac{\partial d_S}{\partial t}(t,x_{\ast}(t))+\langle \nabla_x d_S(t,x_{\ast}(t)), f(x_{\ast}(t),u) \rangle 
\geq  \gamma+\beta +\sigma\quad \ \text{for a.e. } t\in I_{\partial}\ \text{and for all } u\in U
\end{equation}
or 
\begin{equation}\label{m2}\tag{$M_2$}
\frac{\partial d_S}{\partial t}(t,x_{\ast}(t))+\langle \nabla_x d_S(t,x_{\ast}(t)), f(x_{\ast}(t),u) \rangle 
\leq -\sigma \quad \ \text{for a.e. } t\in I_{\partial}\ \text{and for all } u\in U
\end{equation}
hold, where we recall that $d_S(t,x)$ denotes the signed distance between $x$ and $C(t)$. Of course, if $I_\partial=\emptyset$ both
conditions are automatically satisfied.
\begin{rem}
More in general, we can assume that $[0,T]$ can be split into finitely many subintervals such 
that $I_{\partial}$ does not contain their end points and in each subinterval either \eqref{m1} or \eqref{m2} holds. Without loss of generality,
the proofs will be carried out in the case where we have only one interval and either \eqref{m1} or \eqref{m2} hold.
\end{rem}
Before stating the main result of the paper, we recall that in Section \ref{prelim} we have given a meaning to 
\[
\nabla_x d_{C(t)}(x_\ast (t))\qquad\text{and}\qquad \nabla_x^2 d_{C(t)}(x_\ast (t))
\]
also for $t\in I_\partial$.
\begin{thm}\label{necopt}
Assume that $(H_1)$, \ldots, $(H_4)$ hold and consider the minimization problem \eqref{probmin}. Let $(x_\ast,u_\ast)$ be a global minimizer
for which either \eqref{m1} or \eqref{m2} are valid. Then there exist a $BV$ adjoint vector $p:[0,T]\rightarrow \R^n$,
a finite signed Radon measure $\mu$ on $[0,T]$, and measurable vectors $\xi,\eta: [0,T]\rightarrow \R^n$, with $\xi(t)\ge 0$ for $\mu$-a.e. $t$
and $0\le\eta (t)\le \beta+\gamma$ for a.e. $t$, satisfying the following properties:

\smallskip

\noindent $\bullet$ (adjoint equation) \qquad for all continuous functions $\varphi :[0,T]\rightarrow \R^n$
\begin{equation}\label{adj00}
\begin{split}
-\int_{[0,T]} \langle \varphi(t),dp(t) \rangle &= -\int_{[0,T]}\langle \varphi (t),\nabla_x d_{C(t)}(x_\ast (t))\rangle \xi (t)\, d\mu (t)\\
&\qquad - \int_{[0,T]} \langle \varphi (t),\nabla_x^2 d_{C(t)}(x_\ast (t)) p(t) \rangle \eta(t)\, dt\\
&\qquad + \int_{[0,T]} \langle \varphi (t),\nabla_x f(x_\ast (t),u_\ast(t)) p(t) \rangle\, dt, 
\end{split}
\end{equation}

\smallskip

\noindent $\bullet$ (transversality condition) $\qquad -p(T) = \nabla h(x_\ast(T))$,

\smallskip

\noindent $\bullet$ (maximality condition)
\begin{equation}\label{PMP} 
\begin{split}
\langle p(t), \nabla_{u}f(x_{\ast}(t),u_{\ast}(t)) u_{\ast}(t) \rangle = \max_{u\in U}\langle p(t), 
\nabla_{u}f(x_{\ast}(t),u_{\ast}(t))u \rangle \quad \text{for a.e. }t\in [0,T] .
\end{split}
\end{equation}
\end{thm}
Further conditions, in particular on discontinuities of $p$ or, equivalently, on Dirac masses for $\mu$, will be discussed in Proposition
\ref{condp} below. Here we observe only that on $I_0:= [0,T]\setminus I_\partial$, namely on the (possibly empty) set where $x_\ast (t)$
belongs to the interior of $C(t)$, $p$ is absolutely continous and satisfies the classical adjoint equation
\[
-\dot{p}(t) = \nabla_x  f(x_\ast (t),u_\ast (t)) \, p(t)\quad \text{a.e.,}
\]
so that $\mu$, $\xi$, and $\eta$ do not play any role on that set.
This is a simple consequence of \eqref{adj00}.

The proof of Theorem \ref{necopt} is contained in Sections \ref{ACP} and \ref{convergence} and is divided into several propositions, containing
estimates on a sequence of adjoint vectors. Sections \ref{Dynamic} and \ref{MONOD} are devoted to estimates on solutions to
suitable approximations of the primal problem.

\section{Results on the sweeping process and its regularization}\label{Dynamic}
Set $n(t,x)$ to be the unit external normal to $C(t)$ at $x\in \partial C(t)$ and $0$ if $x \in {\rm int}\, C(t)$, and not defined if $x \notin  C(t)$.
Observe that $n(t,x) = \nabla_x d_S (t,x)$ for all $x\notin \text{int}\, C(t)$.

For any solution $x_\ast$ of \eqref{probmin}, 
the general theory on the sweeping process (see, e.g., \cite[Theorem 3.1]{thib}), yields that 
\begin{equation}\label{boundxdot}
\vert \dot{x}_{\ast}(t) \vert\leq  \gamma + \beta\;\text{ for a.e }t\in[0,T].
\end{equation}
The main tool that we are going to use is an approximate control problem, where the dynamics is the Moreau-Yosida regularisation of
\eqref{probmin} and the cost is the original one, plus a penalization term. More precisely, 
the approximate problem is the following one:

\bigskip

For a given $\varepsilon > 0$ and a given admissible control $u_\ast$,
\begin{equation}\label{probeps}
\text{Minimize}  \qquad      h(x(T))+\frac12 \int_{0}^{T}|u(t)-u_{\ast}(t)|^{2}dt
\end{equation}
subject to 
\begin{equation}\label{primeqeps}
\dot{x}(t)\,= \,  -\frac{1}{\varepsilon}\big(x(t)-\text{proj}_{C(t)}(x(t))\big) + f(x(t),u(t)),\quad x(0)=x_{0}\in C(0),
\end{equation}
over all admissible controls $u:[0,T]\rightarrow U$.

\medskip

We label the above problem as $(P_\varepsilon(u_\ast))$. By standard results, for every $\varepsilon > 0$ there exists a global minimizer $u_{\varepsilon}$. 
If $u_\varepsilon$ is such a minimizer
and $x_\varepsilon$ is the solution of \eqref{primeqeps} with $u_\varepsilon$ in place of $u$, we will refer to $(x_\varepsilon,u_\varepsilon)$
as an optimal couple for $(P_\varepsilon(u_\ast))$.

\bigskip

As a preliminary result on $(P_\varepsilon(u_\ast))$, we are going to prove that, thanks to the Lipschitz continuity of the metric projection onto
$C(t)$ on the set $C_\rho$ for each $t$,
and the boundedness of the Lipschitz perturbation $f$, the Cauchy problem \eqref{primeqeps} admits one and only one
solution on the interval $[0,T]$, for every fixed admissible control $u$.
Our first result is in fact concerned with existence, uniqueness and some estimates on such solutions, uniform with respect to $\varepsilon$.
\begin{prop} \label{compactnessepsilon}
Let $C, f, U, h$ be given satisfying assumptions $(H_1)$, \ldots, $(H_4)$. 
Let $\varepsilon_{n}\downarrow0$ and let $\{ u_{n}\}$ be a sequence of admissible controls.
Then, for every $n$ large enough, the problem \eqref{primeqeps} with $\varepsilon_n$ in place of $\varepsilon$ and $u_n$ in place of $u$
admits one and only one solution $x_n$ on the interval $[0,T]$. Such solutions
are Lipschitz uniformly with respect to $n$, with Lipschitz constant $\gamma+2\beta$, and moreover the estimate
\begin{equation}\label{geps}
d_{C(t)}(x_n(t)) \leq \varepsilon_n(\beta+\gamma)\big(1-e^{-t/\varepsilon_n}\big)\leq\varepsilon_n(\beta+\gamma)\quad \text{for all }\, t\in [0,T].
\end{equation}
holds.
\end{prop}
The proof of Proposition \ref{compactnessepsilon} follows the arguments developed in \cite[Section 3]{ML} and will be sketched
after some technical results.
 
First of all, let $x(t)$ be absolutely continuous and set $g(t):= d_{C(t)}(x(t))$. Recalling Lemma 3.1 in \cite{ML}, we have, for a.e. $t\in[0,T]$,
\begin{equation}\label{ST3.1}
\dot{g}(t) g(t) \leq \langle \dot{x}(t),x(t)-\text{proj}_{C(t)}(x(t)) \rangle+\gamma g(t),
\end{equation}
provided
\begin{equation}\label{<rho}
d_{C(t)}(x(t)) < \rho \qquad \text{for all }\, t\in [0,T].
\end{equation}
As an immediate corollary, we obtain the following Lemma.
\begin{lem}\label{Lemma3.3}
For every $\varepsilon  > 0$, let the admissible control $u_\varepsilon$ be given and let $x_{\varepsilon}$ be the corresponding solution of
\eqref{primeqeps}.
Set
\[
g_{\varepsilon}(t)=d_{C(t)}(x_{\varepsilon}(t)),\qquad t\in [0,T]
\]
and assume that $x_\varepsilon$ satisfies \eqref{<rho}.
Then
\begin{equation}\label{gepss}
g_{\varepsilon}(t)\leq \varepsilon(\beta+\gamma)\big(1-e^{-t/\varepsilon}\big)\leq\varepsilon(\beta+\gamma)\quad \text{for all }\, t\in [0,T].
\end{equation}
\end{lem}
\begin{proof}
From \eqref{ST3.1} we obtain 
\[
\dot{g}_{\varepsilon}(t) g_{\varepsilon}(t) \leq \gamma g_{\varepsilon}(t)-\frac{1}{\varepsilon} 
g_{\varepsilon}^{2}(t)+g_{\varepsilon}(t)\vert f(x_{\varepsilon}(t),u_{\varepsilon}(t))\vert,
\]
which yields, if $g_\varepsilon (t) > 0$,
\[
\dot{g}_{\varepsilon}(t) \leq -\frac{1}{\varepsilon} g_{\varepsilon}(t)+\gamma+\beta.
\]
The case $g_\varepsilon (t) = 0$ can be treated exactly as in the proof of \cite[Lemma 3.3]{ML}.
Then the result follows from Gronwall's lemma.
\end{proof}
\begin{proof}[Proof of Proposition \ref{compactnessepsilon}] The proof is divided into two steps. 
First, we assume that the final time $T$ is small enough, namely $0<T\le \theta$,
with
\begin{equation}\label{theta}
 \theta < \frac{\rho}{3(2\beta +\gamma)},
\end{equation}
where we recall that the constants $\rho$, $\beta$, and $\gamma$ appear in the standing assumptions $(H_{1.1})$, $(H_{1.2})$,
and $(H_{3.1})$. Second, the general case will be treated.

Assume now that $T \le \theta$ and let $\varepsilon_n \downarrow 0$ and a sequence $\{ u_n\}$ of admissible controls be given.
It is clear that a solution $x_n$ of \eqref{primeqeps}, with $\varepsilon_n$, resp. $u_n$, in place of $\varepsilon$, resp. $u$,
exists and is defined on its maximal interval of existence $[0,T_n]\subseteq [0,T]$ such that $d_{C(t)}(x_n(t))<\rho$ for all $t\in [0,T_n]$.
We have from Lemma \ref{Lemma3.3} that \eqref{geps} holds on $[0,T_n]$ and so the solution $x_n$ is unique by a standard application of Gronwall's lemma.
It is also easy to see, arguing as in \cite[Sect. 3]{ML}, that the maximal interval of existence must be the whole of $[0,T]$.
Moreover, we have for all $n$
\[
 \vert \dot{x}_{n}(t)-f(x_{n}(t),u_{n}(t))\vert=
\left\vert\frac{x_{n}(t)-\text{proj}_{C(t)}(x_{n}(t))}{\varepsilon_n}\right\vert=\frac{1}{\varepsilon_n}d_{C(t)}(x_{n})
\leq \gamma+\beta \quad\forall t\in [0,T],
\]
from which the conclusion on the Lipschitz constant of $\dot{x}_n$ follows immediately, since $f$ is uniformly bounded by $\beta$. 

Consider now the general case. From the preceding argument, for each $n$ there exists a solution $x_n$ such that,
for all  $t\in [0,\theta]$,
\begin{equation}\label{dxn}
d_{C(t)}(x_n(t)) \le \varepsilon_n (\beta + \gamma) . 
\end{equation}
For every $n$ large enough, we can assume that
\[
 \varepsilon_n (\beta + \gamma) + \theta < \frac{\rho}{3(2\beta +\gamma)}.
\]
Applying the argument used in the preceding step to the Cauchy problem
\[
\begin{cases}
\dot{x}(t) & = -\frac{1}{\varepsilon_n}\big(x(t)-{\rm proj}_{C(t)}(x(t))\big) + f(x(t),u_{n}(t)),
\\
x(0)&=x_n(\theta),
\end{cases}
\]
we can extend $x_n$, keeping the property \eqref{dxn}, up to the time $2\theta$. Since $\theta$ is independent of $n$, the interval
$[0,T]$ can be covered after finitely many steps.
\end{proof}
Our second result is concerned with compactness and passing to the limit for solutions of $(P_\varepsilon(u_\ast))$, as $\varepsilon \to 0$.
\begin{prop}\label{cowlambda}
Let $u_\ast$ be a global minimizer for the problem $(P)$, together with the corresponding solution $x_\ast$ of \eqref{probmin}. 
Let $(x_{\varepsilon},u_{\varepsilon})$ be an optimal couple for the regularised minimization problem $(P_\varepsilon(u_\ast))$.
Then there exists a sequence $\varepsilon_{n}\downarrow0$ such that
\[
\begin{split}
&x_{\varepsilon_{n}}\rightarrow x_{\ast} \text{ weakly in } W^{1,2}([0,T];\R^{n}),\\
&u_{\varepsilon_{n}}\rightarrow u_{\ast}\text{ strongly in }L^{2}([0,T];\R^{m}).
\end{split}
\]
\end{prop}
\begin{proof} 
By Proposition \ref{compactnessepsilon} and assumptions $(H_2)$ and $(H_{3.1})$, we find a sequence $\varepsilon_{n}\downarrow0$ 
and an admissible control $\tilde{u}$ such that
\begin{equation}\label{conv}
\begin{split}
&x_{n}:=x_{\varepsilon_{n}} \text{ converges weakly in } W^{1,2}([0,T];\R^{n}) \text{ to some } x,\\
&u_{n}:=u_{\varepsilon_{n}} \text{ converges weakly in } L^{2}([0,T];\R^{n}) \text{ to } \tilde{u},\\
&\!\!\int_{0}^{T}|u_{n}(t)-u_{\ast}(t)|^{2\, }dt\quad \text{ converges to some }\quad \delta\geq 0,
\end{split}
\end{equation} 
and moreover \eqref{geps} holds for every $n$. Observe that \eqref{geps} implies in turn that $\text{proj}_{C(t)}(x_{n}(t))$ is well defined
and also
\begin{equation}\label{normal}
x_n(t) - \text{proj}_{C(t)}(x_{n}(t))\in N_{C(t)}\big( \text{proj}_{C(t)}(x_{n}(t)) \big)
\end{equation}
for each $t\in [0,T]$ and $n\in \N$.

Let us prove first that there exists an admissible control $\bar{u}$ such that 
\begin{equation}\label{admcont}
\dot{x}(t)\,\in \,  -N_{C(t)}(x(t)) + f(x(t),\bar{u}(t))\; \text{ a.e. on }\, [0,T].
\end{equation}
Indeed, from
\[
-\dot{x}_{n}(t) = \frac{x_{n}(t)-\text{proj}_{C(t)}(x_{n}(t))}{\varepsilon_n} - f(x_{n}(t),u_{n}(t))
\]
and \eqref{ineqphi}, \eqref{normal}, \eqref{geps} it follows immediately that
\begin{equation}\label{nonpos}
\langle -\dot{x}_{n}(t)+f(x_{n}(t),u_{n}(t)),y-\text{proj}_{C(t)}(x_{n}(t))\rangle \leq 
\frac{\gamma+\beta}{2\rho} |y-\text{proj}_{C(t)}(x_{n}(t))|^2  \quad \forall y \in C(t).
\end{equation}
First we see that the uniform convergence of $x_{n}$ to $x$ implies by passing to the limit in \eqref{geps}
that $x(t)\in C(t)$ for all $t\in [0,T]$. Furthermore, by possibly taking a subsequence
we may assume that $z_n:=f(x_{{n}},u_{{n}})$ converges weakly in $L^2([0,T];\R^n)$ to some $z$, and
by Mazur's lemma we can find a convex combination  $\sum_{k=n}^{r(n)}S_{k,n}(-\dot{x}_{k}+z_{k})$,
with $\sum_{k=n}^{r(n)}S_{k,n}=1$ and $S_{k,n}\in[0,1]$ for all $k,n$, which converges strongly in $L^{2}$ and pointwise a.e. to $-\dot{x}+z$.
Let now $t\in [0,T]$ and $y \in C(t)$. We have 
\[
\begin{split}
 \langle -\dot{x}(t)+z(t),y-x(t)\rangle & = \big\langle -\dot{x}(t)+z(t)- \sum_{k=n}^{r(n)}S_{k,n}(-\dot{x}_{k}(t)+z_{k}(t)),y-x(t)\big\rangle\\
&\qquad + \sum_{k=n}^{r(n)}S_{k,n} \langle -\dot{x}_{k}(t)+z_{k}(t),y-\text{proj}_{C(t)}(x_{k}(t))\rangle 
\\ 
 &\qquad +\sum_{k=n}^{r(n)}S_{k,n}  \langle -\dot{x}_{k}(t)+z_{k}(t),-x(t)+\text{proj}_{C(t)}(x_{k}(t))\rangle .
\end{split}
\]
The first and the third summands in the above expression tend to zero a.e. The second one, thanks to \eqref{nonpos}, satisfies the estimate
\[
 \sum_{k=n}^{r(n)}S_{k,n} \langle -\dot{x}_{k}(t)+z_{k}(t),y-\text{proj}_{C(t)}(x_{k}(t))\rangle \le 
\frac{\gamma+\beta}{2\rho} \sum_{k=n}^{r(n)}S_{k,n} |y-\text{proj}_{C(t)}(x_{k}(t))|^2.
\]
Thus, passing to the limit one obtains
\[
\langle -\dot{x}(t)+z(t),y-x(t)\rangle  \le \frac{\gamma+\beta}{2\rho} |y-x(t)|^2\qquad\forall y\in C(t).
\]
This proves that $\dot{x}(t)\in -N_{C(t)}(x(t)) + z(t)$ for a.e. $t\in [0,T]$. 
Since $f(x,U)$ is convex for all $x$, from the classical Convergence Theorem
(see, e.g., \cite[Theorem 1, p. 60]{AuCe}) it follows that $z(t)\in f(x(t),U)$ for a.e. $t$. It then follows from 
from Filippov's Selection Theorem (see, e.g., \cite[Th. 2.3.13]{vinter}) that there exists $\bar{u}(\cdot)$ such that $z(t)=f(x(t),\bar{u}(t))$.
This proves \eqref{admcont}.

We claim now that $(x,\bar{u})=(x_{\ast},u_{\ast})$.
To this aim, define $x_{\ast}^{n}$ to be the unique solution of the Cauchy problem
\[
\left\{ \begin{array}{l}
\dot{y}(t)\,= \,  -\frac{1}{\varepsilon_{n}}\big(y(t)-\text{proj}_{C(t)}(y(t))\big) + f(y(t),u_{\ast}(t)),
\\
y(0)=x_{0}\;,
\\
\end{array}\right.
\]
on $[0,T]$ and observe that $x_{\ast}^{n}$ converges weakly to $x_{\ast}$ in $W^{1,2}([0,T];\R^{n})$ (see \cite[Lemma 3.6]{ML}).
Since $(x_{n},u_{n})$ is an optimal couple (namely, a global minimizer) for $(P_\varepsilon(u_\ast))$, we have 
\begin{equation}\label{xnstar}
 h(x_{\ast}^{n}(T)) \geq h(x_{n}(T))+\frac12 \int_{0}^{T}|u_{n}(t)-u_{\ast}(t)|^{2}dt 
\end{equation}
for all $n\in \N$.
By passing to the limit in \eqref{xnstar}, using the weak lower semicontinuity of the integral together with 
our convergence properties \eqref{conv}, we obtain
\[
 h(x_{\ast}(T))  \geq h(x(T))+ \delta
\geq h(x(T))+ \int_{0}^{T}|\tilde{u}(t)-u_{\ast}(t)|^{2}dt.
\]
Since $x_\ast$ is a global minimizer for the problem $(P)$, the above inequalities
imply that $\delta=0$, i.e., $u_{n}\rightarrow u_{\ast}$ strongly in $L^{2}([0,T];\R^{n})$.

Now, the strong convergence of $u_{n}$ to $u_{\ast}$ allows us to prove that $x_{n}\rightarrow x_{\ast}$ weakly in $W^{1,2}([0,T];\R^{n})$.
Indeed, set $r_{n}=\frac{1}{2}|x_{n}(t)-x_{\ast}(t)|^{2}$. Then, by using the fact that 
$-\dot{x}_n(t) + f(x_n(t),u_n(t))=\frac{1}{\varepsilon_{n}}\big(x_{n}(t)-\text{proj}_{C(t)}(x_n(t))\big)\in N_{C(t)}(x_{n}(t))$
for all $t\in [0,T], n\in\N$ together with the (hypo)monotonicity of the normal cone to $C(t)$ -- which follows immediately
from \eqref{ineqphi} -- and $(H_{3.3})$, \eqref{boundxdot}, we obtain, for all $t\in [0,T]$ and each $n\in \N$ large enough,
\[
\begin{split}
 \dot{r}_{n}(t)& = \langle -\dot{x}_{n}(t)+\dot{x}_{\ast}(t), - x_{n}(t)+x_{\ast}(t) \rangle   
\\
 &\leq - \langle f(x_{n}(t),u_{n}(t))-f(x_{\ast}(t),u_{\ast}(t)), x_{n}(t)-x_{\ast}(t) \rangle + \frac{2\gamma + 3\beta}{\rho} r_n(t)
 \\
 & \leq  K r_{n}(t)+k|x_{n}(t)-x_{\ast}(t)|\,  |u_{n}(t)-u_{\ast}(t)|,
\end{split}
\]
where $K := k +\frac{2\gamma + 3\beta}{\rho} $.
Since $r_{n}(0)=0$ for each $n \in \N$, Gronwall's Lemma yields, for each $t\in [0,T]$,
\[
r_{n}(t)\leq k\int_{0}^{T} |x_{n}(t)-x_{\ast}(t)|\, |u_{n}(t)-u_{\ast}(t)| e^{K(T-t)}dt,
\]  
which, by the strong convergence of $u_{n}$ to $u_{\ast}$, implies that $x(t)=x_{\ast}(t)$ for all $t\in [0,T]$.
\end{proof}
\begin{rem}
Observe that a similar argument implies that the sequence $\lbrace x_{n}\rbrace$ is Cauchy for the uniform convergence.
\end{rem}

\section{Monotonicity of the distance} \label{MONOD}
Let $x_\ast$ be a given optimal trajectory, and $x_n$
be trajectories of the regularized dynamics, and recall that $I_{\partial}:= \lbrace t\in [0,T]:x_{\ast}(t)\in  \partial C(t)\rbrace$.
The first result in this section will be crucial in order to allow some estimates involving $\nabla_{x}d_{C(t)}(x_n(t))$ 
and $\nabla_{x}^{2}d_{C(t)}(x_n(t))$ by forbidding that $x_n(t)$ remains on $\partial C(t)$ 
on a set of times with positive measure. Recall that $\partial C(t)$ is the discontinuity set of $\nabla_{x}d_{C(t)}(x)$
and $\nabla_{x}^{2}d_{C(t)}(x)$ as a function of $x$.
The simplest way to satisfy this requirement is giving sufficient conditions in order to let 
$\frac{d}{dt}d_S(t,x_n(t))$ be nonzero for a.e. $t$ in a suitable neighborhood of $I_{\partial}$.
\begin{prop}\label{monot}
Assume $(H_1)$, \ldots, $(H_4)$, and let the admissible control $u_\ast$, with the corresponding solution $x_\ast$ of \eqref{probmin},
be such that that there exists $\sigma >0$ for which either \eqref{m1} or \eqref{m2} hold.
Let $\varepsilon_{n}\downarrow 0$ and let $u_n$ be admissible controls such that $u_{{n}}\rightarrow u_{\ast}$ 
strongly in $L^{2}([0,T];\R^{m})$ and the corresponding solutions of \eqref{primeqeps} $x_n\rightarrow x_{\ast}$ strongly in $W^{1,2}([0,T];\R^{n})$. 
Then for each $n$ large enough there exists at most one time $t_{n}\in [0,T]$ such that $x_n(t_{n})\in  \partial C(t_{n})$.
\end{prop}
\begin{proof}
We write the proof for the case \eqref{m1}, the case \eqref{m2} being similar and easier.

Assume (\ref{m1}) and let $\delta > 0$ be such that for a.e. $t\in [0,T]$ with $ d(t,I_{\partial}) < \delta$ 
and all $x, y \in \R^{n} $ with $\vert x-x_{\ast}(t)\vert$, $\vert y-y_{\ast}(t)\vert$ , $u\in U$ we have
\begin{equation}\label{m1xy}
\frac{\partial d_S}{\partial t}(t,x)+\langle \nabla_{x}d_S(t,y), f(x,u) \rangle-(\gamma+\beta)\geq \frac{\sigma}{2} .
\end{equation}
Let $y_{n}(t)$ be the projection of $x_n(t)$ onto $C(t)$. Then, for a.e. $t \in [0,T]$ we have
\begin{equation}  \label{ineqg} 
\begin{split}
 \frac{d}{dt}d_S(t,x_n(t)) &= \frac{\partial d_S}{\partial t}(t,x_n(t))
+\langle \nabla_{x}d_S(t,x_n(t)), \dot{x}_{\varepsilon_{n}}(t) \rangle   
\\
 & = \frac{\partial d_S}{\partial t}(t,x_n(t))+ \Big\langle \nabla_{x}d_S(t,y_{n}(t))
+\nabla_{x}^{2}d_S(t,y_{n}(t))(x_n(t)-y_{n}(t))
 \\
&  \quad +\sum_{\vert \alpha \vert = 2 , \alpha \in \N^{n}}\int_{0}^{1}(1-\tau)\partial_{x}^{\alpha}
\nabla_{x}d_S(t,y_{n}(t)+\tau(x_n(t)-y_{n}(t))(x_n(t)-y_{n}(t))^{\alpha}\, d\tau ,\\
& \qquad \qquad -\frac{x_n(t)-y_{n}(t)}{\varepsilon_{n}}+ f(x_n(t),u_n(t))\Big\rangle ,
\end{split}
\end{equation}
where $\alpha$ is the multiindex $(\alpha_1,\ldots ,\alpha_n)\in\mathbb{N}^n$, 
$\partial^\alpha_x = \frac{\partial^{\alpha_1}}{\partial x_1^{\alpha_1}}\ldots 
\frac{\partial^{\alpha_n}}{\partial x_n^{\alpha_n}}$, and $|\alpha|$ denotes the sum of all entries of $\alpha$.
Observe that in the above expression all summands involving higher order partial derivatives of $d_S$
vanish if $x_n(t)\in C(t)$, since in this case $y_n(t)=x_n(t)$.

Thanks to \eqref{geps}, for all $n$ large enough we have 
\begin{center}
$ d_{C(t)}(x_n(t))\leq \varepsilon_{n}(\gamma +\beta )< \frac{\delta}{2}$.
 \end{center}
Therefore, for all  $t\in I_{ \partial } + (- \delta , \delta )$, we obtain from \eqref{ineqg} that
\[
 \frac{d}{dt}d_S(t,x_n(t)) \geq \frac{\partial d_S}{\partial t}(t,x_n(t))
-(\gamma+\beta)+\langle \nabla_{x}d_S(t,y_{n}(t)), f(x_n(t),u_n(t)) \rangle - K d_{C(t)}(x_n(t))  ,
\]
where $K$ is independent of $t$ and $n$.
Since $x_n$ converges uniformly to $x_\ast$, we obtain from \eqref{m1xy} that $\frac{d}{dt}d_S(t,x_n(t))> 0$ 
for a.e. $t\in I_{ \partial } + (- \delta , \delta )$, for all $n$ large enough,
whence there exists at most one $t \in [0,T] \setminus (I_{ \partial } + (- \delta , \delta ))$ such that 
$x_n(t)\in  \partial C(t)$. Since in $[0,T]\setminus (I_{ \partial } 
+ (- \delta , \delta ))$ the trajectory $x_n(t)$ belongs to ${\rm int}\, C(t)$ for all $n$ large enough,
no further crossings of $\partial C(t)$ are possible.

The second case is analogous and actually easier. In fact there will not be any crossing of $\partial C(t)$ on $(0,T]$, 
for all $n$ large enough, since on a suitable neighborhood of $I_{\partial}$ we will have $\frac{d}{dt}d_S(t,x_n(t))< 0$. 
\end{proof}
The following simple corollaries will be useful in the discussion of necessary conditions. 
\begin{prop}\label{Im1}
Assume \eqref{m1}. Then $I_{\partial}$ is an interval and, if it is nonempty, $\sup I_{\partial}=T$.
\end{prop}
\begin{proof}
It is enough to show that if $t\in I_{\partial}$, then $[t,T] \subset I_{\partial}$. 
To this aim, assume by contradiction that there exists $t<T$ such that $t\in I_{\partial}$, but $\bar{t}:= \sup I_{\partial}< T$ . 
This means, in particular, that for all $s\in (\bar{t},T]$ we have $ d_S(s,x_{\ast}(s))< 0$ .
Thus, for all such $s$ we have  
\[
\begin{split}
 0> d_S(s,x_{\ast}(s))- d_S(\bar{t},x_{\ast}(\bar{t}))& = \int_{\bar{t}}^{s}\Big( \frac{\partial d_S}{\partial t}
(\tau,x_{\ast}(s))+\langle \nabla_{x}d_S(\tau,x_{\ast}(\tau)), \dot{x}_{\ast}(\tau) \rangle\Big)\, d\tau   
\\
 & = \int_{\bar{t}}^{s}\Big( \frac{\partial d_S}{\partial t}(\tau,x_{\ast}(s))+\langle \nabla_{x}d_S(\tau,x_{\ast}(\tau)), 
f(x_{\ast}(\tau),u_{\ast}(\tau))\rangle \Big)\, d\tau
 \end{split}
\]
and the integrand is positive if $s$ is close enough to $\bar{t}$, a contradiction.
\end{proof}
\begin{prop}\label{Im2}
Assume \eqref{m2}. Then $I_{\partial}$ is at most the singleton $\lbrace 0 \rbrace$ . 
\end{prop}
\begin{proof}
Assume by contradiction that there exists $\bar{t}> 0$, with $\bar{t}\in I_{\partial}$. Then, for all $t<\bar{t}$ we have, on one hand,
\[
d_S(\bar{t},x_{\ast}(\bar{t}))- d_S(t,x_{\ast}(t))\geq 0,
\] 
while on the other, 
\[
\begin{split}
d_S(\bar{t},x_{\ast}(\bar{t}))- d_S(t,x_{\ast}(t))& = \int_{t}^{\bar{t}}\Big( \frac{\partial d_S}{\partial t}(\tau,x_{\ast}(s))
+\langle \nabla_{x}d_S(\tau,x_{\ast}(\tau)), \dot{x}_{\ast}(\tau) \rangle\Big)\, d\tau   
\\
 & = \int_{t}^{\bar{t}}\Big( \frac{\partial d_S}{\partial t}(\tau,x_{\ast}(s))+\langle 
\nabla_{x}d_S(\tau,x_{\ast}(\tau)),\dot{x}_{\ast}(\tau)- f(x_{\ast}(\tau),u_{\ast}(\tau))\rangle 
 \\
 & \qquad + \langle \nabla_{x}d_S(\tau,x_{\ast}(\tau)), f(x_{\ast}(\tau),u_{\ast}(\tau))\rangle \Big)\, d\tau
 \\
 \end{split}
\]
Observe that if $x_{\ast}(\tau) \in {\rm int}\, C(\tau)$, then $ \dot{x}_{\ast}(\tau)- f(x_{\ast}(\tau),u_{\ast}(\tau)) = 0$, 
while if $x_{\ast}(\tau) \in \partial C(\tau)$, then 
$\dot{x}_{\ast}(\tau)- f(x_{\ast}(\tau),u_{\ast}(\tau))= -\delta (\tau) \nabla_{x}d_S(\tau,x_{\ast}(\tau))$
for a bounded nonnegative function $\delta$. Therefore, \eqref{m2} implies that the integrand is $< 0$, provided $t$ 
is close enough to $\bar{t}$, yielding a contradiction.  
\end{proof}

\section{The Approximate Control Problem}\label{ACP}
Given a global minimizer $u_\ast$ of the problem $(P)$ and $\varepsilon >0$, we recall that in Section \ref{Dynamic}
the approximate problem $(P_\varepsilon (u_\ast))$ was defined and studied. 

Let $u_\varepsilon$ be a global minimizer for $(P_\varepsilon (u_\ast))$.
By Proposition \ref{cowlambda}, we know that, up to a subsequence, $u_\varepsilon$ converges weakly in $L^2([0,T];\R^m)$ to $u_\ast$ and
$x_\varepsilon$ converges strongly in $W^{1,2}([0,T];\R^n)$ to the optimal trajectory $x_\ast$, namely the trajectory of \eqref{probmin}
where we replace $u$ with $u_\ast$.

In order to state necessary conditions satisfied by $(x_{\varepsilon},u_{\varepsilon})$, we recall that the 
subdifferentiability of the Lipschitz map 
\begin{center}
$x\mapsto x-\text{proj}_{C(t)}(x) := P(t,x)$
\end{center}
was preliminarly discussed in Section \ref{prelim}. 
Here we recall only that under our standing assumptions the unit external normal to $C(t)$ at $x\in\partial C(t)$ is $\nabla_x d_S(t,x)$.

Differently from classical computations in control theory, we will not consider needle variations, but rather compute a 
directional derivative of the cost, following \cite{BPK}.
More precisely, let $(x_{\varepsilon},u_{\varepsilon})$ be an optimal pair for the problem \eqref{probeps} 
subject to \eqref{primeqeps} and let $\tilde{u}$ be an admissible control. For $\sigma \in [0,1]$ set $u_{\sigma}(t)=
u_{\varepsilon}(t)+\sigma(\tilde{u}(t)-u_{\varepsilon}(t))$ and observe that, thanks to the convexity of the control set $U$, 
this is an admissible control as well. Set $x_{\sigma}$ to be the corresponding solution of \eqref{primeqeps}.
We wish to compute the directional derivative of the cost $ J(x,u;u_{\ast})$ at $(x_{\varepsilon},u_{\varepsilon})$ 
in the direction $\tilde{u}-u_{\varepsilon}$. Should this derivative exist, then it would be $\geq 0$, 
by the minimality of $(x_{\varepsilon},u_{\varepsilon})$, namely  
\begin{equation*}
 \lim_{\sigma \rightarrow 0}\frac{ J(x_{\sigma},u_{\sigma};u_{\ast})- J(x_{\varepsilon},u_{\varepsilon};u_{\ast})}{\sigma}\geq 0.
\end{equation*}
The difference quotient in the above expression consists of the summands 
\[
\frac{h(x_{\sigma}(T))-h(x_{\varepsilon}(T))}{\sigma}\; +\;  
\frac{1}{2\sigma}\int_{0}^{T}(|u_{\sigma}(t)-u_{\ast}(t)|^{2}-|u_{\varepsilon}(t)-u_{\ast}(t)|^{2})\, dt.
\]
The limit of the second summand for $\sigma \to 0$ is straightforward and equals 
\[
 \int_{0}^{T}\langle \tilde{u}(t)-u_{\varepsilon}(t),u_{\varepsilon}-u_{\ast}(t)\rangle\, dt
\]
The limit of the first summand is 
\begin{equation*}
\lim_{\sigma \rightarrow 0} \left(\Big\langle \nabla h(x_{\varepsilon}(T)), \frac {x_{\sigma}(T)- 
x_{\varepsilon}(T)}{\sigma} \Big\rangle + \frac {o\big(x_{\sigma}(T)- x_{\varepsilon}(T)\big)}{\sigma}\right).
\end{equation*}
Therefore, we are lead to compute the limit 
\[
 \lim_{\sigma \rightarrow 0}\frac{ x_{\sigma}(T)-x_{\varepsilon}(T)}{\sigma}.
\]
This is classical (see, e.g., \cite[Theorem 3.4]{ABBP}) under the assumption of continuous differentiability with respect of both $x$ and $u$ 
of the right hand side of \eqref{primeqeps} for a.e. $t$. Such assumptions are valid in our setting thanks to Proposition \ref{monot}, since
there exists at most one time $t$ such that $x_\varepsilon (t)$ belongs to $\partial C(t)$.
Therefore, thanks to Theorem 3.4 in \cite{ABBP} we have 
\begin{equation}\label{gradh}
\lim_{\sigma \rightarrow 0}\frac{ x_{\sigma}(T)-x_{\varepsilon}(T)}{\sigma}= \big\langle \nabla h(x_{\varepsilon}(T)), 
\int_{0}^{T}M(T,t) \nabla_{u}f(x_{\varepsilon}(t),u_{\varepsilon}(t)) (\tilde{u}(t)-u_{\varepsilon}(t))\, dt\big\rangle  ,
\end{equation}
where $M(T,t)$ is the fundamental matrix solution of the linear O.D.E.
\[
\dot{v}(t) 
  = \Big(\frac{-1}{2\varepsilon}\nabla_{x}^{2}d^{2}(x_{\varepsilon}(t),C(t))+\nabla_{x} f(x_{\varepsilon}(t),u_{\varepsilon}(t))\Big) v(t)
\]
We can write the right hand side of \eqref{gradh} as 
\[
 \int_{0}^{T} \big\langle M^{T}(T,t) \nabla h(x_{\varepsilon}(T)), 
\nabla_{u}f(x_{\varepsilon}(t),u_{\varepsilon}(t)) (\tilde{u}(t)-u_{\varepsilon}(t)) \big\rangle \, dt  .
\]
By setting $p_{\varepsilon}(t)$ to be the solution of the adjoint equation 
\begin{equation}\label{adjeps}
\begin{cases}
-\dot{p}_{\varepsilon}(t)&= \big(\frac{-1}{2\varepsilon}\nabla_{x}^{2}d^{2}(x_{\varepsilon}(t),C(t))
+\nabla_{x} f(x_{\varepsilon}(t),u_{\varepsilon}(t))\big)p_{\varepsilon}(t)  ,\quad t\in [0,T]\\
-p_{\varepsilon}(T)&=\nabla h(x_{\varepsilon}(T)) ,
\end{cases}
\end{equation}
the sign condition on the directional derivative becomes, for all feasible $\tilde{u}$, 
\[
\int_{0}^{T}  \big( \langle -p_{\varepsilon}(t), \nabla_{u}f(x_{\varepsilon}(t),u_{\varepsilon}(t)) (\tilde{u}(t)-u_{\varepsilon}(t)) \rangle 
+ \langle u_{\varepsilon}(t)-u_{\ast}(t),\tilde{u}(t)-u_{\varepsilon}(t) \rangle \big)\, dt \geq 0.
\]
Since $\tilde{u}$ is an arbitrary measurable selection from $U$, we obtain
\begin{multline}\label{PMPeps} 
\langle p_{\varepsilon}(t), \nabla_{u}f(x_{\varepsilon}(t),u_{\varepsilon}(t)) u_{\varepsilon}(t) \rangle 
- \langle u_{\varepsilon}(t)-u_{\ast}(t),u_{\varepsilon}(t) \rangle  =\\
 \max_{u\in U}\{ \langle p_{\varepsilon}(t), 
\nabla_{u}f(x_{\varepsilon}(t),u_{\varepsilon}(t))u \rangle - \langle u_{\varepsilon}(t)-u_{\ast}(t),u \rangle\}\qquad 
 \text{for a.e. } t\in [0,T] .
\end{multline}
Now we recall that, thanks to Proposition \eqref{monot}, if $\varepsilon_{n}\downarrow 0$ is such that $(x_{\varepsilon_{n}},u_{\varepsilon_{n}})
\rightarrow (x_{\ast},u_{\ast})$, for all $n$ large enough the right hand side of \eqref{adjeps} is single valued except for at most one 
point $t$ and equals either $\nabla_{x}f(x_{\varepsilon_{n}}(t),u_{\varepsilon_{n}}(t))
p_{\varepsilon_{n}}(t)$, if $x_{\varepsilon_{n}}(t)\in  {\rm int}\, C(t)$, or
\[
\Big( \frac{-1}{\varepsilon_{n}}d_{C(t)}(x_{\varepsilon_{n}})\nabla_{x}^{2}d_{C(t)}(x_{\varepsilon_{n}})+ 
\nabla_{x}d_{C(t)}(x_{\varepsilon_{n}})\otimes \nabla_{x}d_{C(t)}(x_{\varepsilon_{n}})+ 
\nabla_{x}f(x_{\varepsilon_{n}}(t),u_{\varepsilon_{n}}(t))\Big)p_{\varepsilon_{n}}(t) 
\]
if $x_{\varepsilon_{n}}\not\in C(t)$.
This means that we can consider \eqref{adjeps} as a differential equation with a switch, 
which occurs (at most) at a single time $t_{\varepsilon_{n}}$. Without loss of generality, we can assume that 
\begin{equation}\label{convtn}
t_{\varepsilon_n} \quad\text{converges to some }\bar{t} \in I_{\partial}.
\end{equation}
  
\section{Estimates and convergence for the adjoint vectors}\label{convergence}
In this section we keep the notations of the preceding one, and we consider a sequence $\varepsilon_n\downarrow 0$ such that\
the conclusions of Propositions \ref{cowlambda} and \ref{monot} are valid.
\subsection{Estimates}
First we prove some uniform estimates on the sequence $\lbrace p_{\varepsilon_n}(\cdot) \rbrace$, which will guarantee compactness in a suitable space.
\begin{lem}\label{lemma1}
The sequence $\lbrace p_{\varepsilon_n}(\cdot) \rbrace$ is uniformly bounded in $L^{\infty}([ 0, T] ;\R^{n})$. 
\end{lem}
\begin{proof}
We can rewrite \eqref{adjeps} as
\[
-\dot{p}_{\varepsilon_n}(t)= \frac{-1}{\varepsilon_n} \nabla_{x} P(t,x_{\varepsilon_n}(t))p_{\varepsilon_n}(t) 
+\nabla_{x} f(x_{\varepsilon_n}(t),u_{\varepsilon_n}(t)) p_{\varepsilon_n}(t),
\]
where the Jacobian $\nabla_xP(t,x_{\varepsilon_n}(t))$ of $P(t,x_{\varepsilon_n}(t))$ with respect to the state variable $x$ exists for all 
$t$ different from the (possible) switching time $t_{\varepsilon_n}$.
Recalling \eqref{defP} and \eqref{D2d}, we have, for all $t\neq t_{\varepsilon_n}$,
\[
\begin{split}
\nabla_{x} P(t,x_{\varepsilon_n}(t))\, p_{\varepsilon_n}(t)&= 
 d_{C(t)}(x_{\varepsilon_n}(t))\nabla_{x}^{2}d_{C(t)}(x_{\varepsilon_n}(t))p_{\varepsilon_n}(t) \\
&\qquad + \langle \nabla_{x}d_{C(t)}(x_{\varepsilon_n}(t)),p_{\varepsilon_n}(t)\rangle \nabla_{x}d_{C(t)}(x_{\varepsilon_n}(t)).
\end{split}
\]
From this we obtain
\begin{equation*} 
\begin{split}
-\dot{p}_{\varepsilon_n}(t)& = \frac{-d_{C(t)}(x_{\varepsilon_n}(t))}{\varepsilon_n}\nabla_{x}^{2}d_{C(t)}(x_{\varepsilon_n}(t))
p_{\varepsilon_n}(t)-\frac{1}{\varepsilon_n}\langle \nabla_{x}d_{C(t)}(x_{\varepsilon_n}(t)),p_{\varepsilon_n}(t)\rangle 
\nabla_{x}d_{C(t)}(x_{\varepsilon_n}(t))
\\
&\qquad + \nabla_{x} f(x_{\varepsilon_n}(t),u_{\varepsilon_n}(t)) p_{\varepsilon_n}(t).
\end{split}
\end{equation*}
Observe that either $p_{\varepsilon_n}$ vanishes identically or it is never zero. In this case, we multiply both sides of the adjoint equation by 
$\frac{p_{\varepsilon_n}(t)}{\vert p_{\varepsilon_n}(t) \vert}$, and obtain 
\begin{equation}\label{rrrprox}
\begin{split}
-\frac{d}{dt}\vert p_{\varepsilon_n}(t) \vert & =\frac{-d_{C(t)}(x_{\varepsilon_n}(t))}{\varepsilon_n}
\big\langle \frac{p_{\varepsilon_n}(t)}{\vert p_{\varepsilon_n}(t) \vert},\nabla_{x}^{2}d_{C(t)}(x_{\varepsilon_n}(t))
 p_{\varepsilon_n}(t)\big\rangle\\
&\qquad -\frac{1}{\varepsilon_n}\frac{1}{|p_{\varepsilon_n}(t)|}\big\langle \nabla_{x}d_{C(t)}(x_{\varepsilon_n}(t)),p_{\varepsilon_n}(t)\big\rangle^{2}
\\
&\qquad + \big\langle \nabla_{x} f(x_{\varepsilon_n}(t),u_{\varepsilon_n}(t)) p_{\varepsilon_n}(t), 
\frac{p_{\varepsilon_n}(t)}{\vert p_{\varepsilon_n}(t) \vert}\big\rangle .
\end{split}
\end{equation}
The second term on the right hand side of \eqref{rrrprox} is nonpositive, while the first one is bounded by Lemma \ref{Lemma3.3}, 
recalling that, if $\varepsilon_n$ is small enough and $t\neq t_{\varepsilon_n}$, $x_{\varepsilon_n} (t)$ belongs to a set where 
$d_{C(t)}(\cdot)$ is $C^{1,1}$. Let $c$ be a Lipschitz constant for $\nabla_x d_{C(t)}(\cdot)$ on this set.
Integrating over the interval $[t,T]$ and recalling the final time condition contained in \eqref{adjeps} yields
\[ 
\vert p_{\varepsilon_n}(t)\vert-\vert \nabla h(x_{\varepsilon_n}(T))\vert  \leq \int_{t}^{T}c(\gamma+\beta)
\vert p_{\varepsilon_n}(s)\vert ds+\int_{t}^{T} \langle \nabla_{x} f(x_{\varepsilon_n}(t),u_{\varepsilon_n}(t)) p_{\varepsilon_n}(t), 
\frac{p_{\varepsilon_n}(t)}{\vert p_{\varepsilon_n}(t) \vert}\rangle ds,
\]
whence
\[
\vert p_{\varepsilon_n}(t)\vert \leq \vert \nabla h(x_{\varepsilon_n}(T))\vert + \big(c(\gamma+\beta)+k\big)
\int_{t}^{T}\vert p_{\varepsilon_n}(s)\vert \, ds,
\]
recalling that $f$ is $k$-Lipschitz continuous.
Now Gronwall's Lemma in integral form yields
\[    
\vert p_{\varepsilon_n}(t)\vert \leq \vert \nabla h(x_{\varepsilon_n}(T))\vert e^{(c(\gamma+\beta)+k)\, (T-t)}\; \text{ for all } t\in [0,T].
\]
Since $x_{\varepsilon_n}(T)$ converges uniformly to $x_{\ast}(T)$ and $h$ is of class $ C^{1}$, the proof is concluded.
\end{proof}
We deal now with a uniform $L^1$-bound for $\{\dot{p}_{\varepsilon_n}\}$. 
For simplicity of notation, we set $t_{n}:= t_{\varepsilon_{n}}$, $x_{n}(t):= x_{\varepsilon_{n}}(t)$, $u_{n}(t):= u_{\varepsilon_{n}}(t)$, $p_{n}(t):= 
p_{\varepsilon_{n}}(t)$, $\delta_{n} (t):= d_{C(t)}(x_{n}(t))$, $\delta'_n(t):=\nabla_x d_{C(t)}(x_n(t))$,
and finally $\delta''_n(t):=\nabla_x^2 d_{C(t)}(x_n(t))$.
We observe first that, thanks to \eqref{geps} and the fact that $x_n(t)$, for all $t\neq t_n$, 
belongs to a region where $d_{C(t)}(\cdot)$ is
of class $\mathcal{C}^{1,1}$, 
\begin{equation}\label{bddn}
\frac{\delta_{n}(\cdot)} {\varepsilon_{n}} \delta''_{n}(\cdot)\quad \text{is well defined and is bounded in } L^{\infty}([ 0, T] ;\R^{n}), 
\text{ uniformly with respect to $n$.}
\end{equation}
Define now the normal component of $p_{n}(t)$ as
\[
\xi_{n}(t)= \langle p_n(t), \nabla_x d_{C(t)}(x_n(t))\rangle \; (= \langle p_{n}(t) ,  \delta'_{n}(t)\rangle),\quad t\neq t_n.
\]
We have, for a.e. $t$ (in particular $t \neq t_{n} $),
\begin{equation}\label{xidot}
\dot{\xi}_{n}(t)= \langle \dot{p}_{n}(t) , \delta'_{n}(t)\rangle + \langle p_{n}(t) , 
\delta_n''(t)\dot{x}_{n}(t) \rangle.
\end{equation}
With this notation, the primal dynamics in \eqref{primeqeps} and the dual one in \eqref{adjeps} can be rewritten, respectively, as
\[
\begin{split}
\dot{x}_n(t) &= -\frac{\delta_n(t)}{\varepsilon_n}\delta'_n(t) + f(x_n(t),u_n(t))\\
-\dot{p}_n(t) &= -\frac{\delta_n(t)}{\varepsilon_n}\delta_n''(t)p_n(t) -\frac{1}{\varepsilon_{n}}  \delta'_{n}(t)\otimes  \delta'_{n}(t) p_{n}(t)
                     + \nabla_x f(x_n(t),u_n(t))p_n(t)\\
&= -\frac{\delta_n(t)}{\varepsilon_n}\delta_n''(t)p_n(t) -\frac{\xi_n(t)}{\varepsilon_n}\delta_n'(t) + \nabla_x f(x_n(t),u_n(t))p_n(t). 
\end{split}
\]
By inserting $\dot{p}_n(t)$ and $\dot{x}_n(t)$ from the above equations into \eqref{xidot}, we obtain, for a.e. $t\in [0,T]$,
\begin{equation}\label{eqxin}
\begin{split}
- \dot{\xi}_{n}(t)& = -\frac{\delta_{n}(t)} {\varepsilon_{n}} \langle  \delta''_{n}(t) p_{n}(t), \delta'_{n}(t) 
\rangle -\frac{\xi_n(t)}{\varepsilon_{n}}|\delta_n'(t)|^2
\\
&\qquad + \langle \nabla_{x} f(x_{n}(t),u_{n}(t)) p_{n}(t),  \delta'_{n}(t)\rangle\\
&\qquad - \Big\langle p_{n}(t) ,
 \delta''_{n}(t) \Big(-\frac{\delta_{n}(t)} {\varepsilon_{n}} \delta'_{n}(t)+f(x_{n}(t),u_{n}(t)) \Big)\Big\rangle.
\end{split}
\end{equation}
In order to simplify the above relation, we observe that,
\[
\delta''_n(t) \delta'_n(t)=\nabla_{x}^{2}d_{C(t)}(x) \nabla_{x}d_{C(t)}(x)= 0 \quad \text{for all $t\in [0,T]$ and all $x\notin \partial C(t)$}
\]
(see Lemma 3.8 in \cite{BPK}), and furthermore that
\[
\xi_n(t) = \xi_n(t) |\delta_n'(t)|^2\quad \text{for all }\, t\in [0,T], t\neq t_n,
\]
since if $x_n(t)\in \text{int}\, C(t)$ then both sides are zero, while if $x_n (t)\not\in C(t)$ we have $|\delta_n'(t)|=1$.
Therefore, recalling that $x_n(t)\not\in \partial C(t)$ for all $t\neq t_n$,
the equation \eqref{eqxin} becomes, for a.e. $t\in [0,T]$ (with $t\not= t_n$),
\begin{equation}\label{relatxi}
\begin{split}
- \dot{\xi}_{n}(t)+ \dfrac{1}{\varepsilon_{n}}\xi_{n}(t)& = -\frac{\delta_{n}(t)}{\varepsilon_{n}} 
\Big[  \langle  \delta''_{n}(t) p_{n}(t), \delta'_{n}(t) \rangle 
+ \langle p_{n}(t) , \delta''_{n}(t) f(x_{n}(t),u_{n}(t)) \rangle \Big] 
\\
&\qquad + \langle \nabla_{x}f(x_{n}(t), u_{n}(t)) p_{n}(t), \delta'_{n}(t) \rangle.
\end{split}
\end{equation}
Observe now that the right hand side of the above equality is bounded in $L^{\infty}([ 0, T] ,\R^{n})$, uniformly with respect to $n$, 
thanks to Lemma \ref{lemma1} and \eqref{geps}, and to the $\rho$-prox-regularity of the moving set $C(\cdot)$. 
Observe also that Lemma \ref{lemma1} implies that $\xi_n(t)$ is bounded in $L^\infty(0,T)$, uniformly with respect to $n$.
By multiplying both sides of \eqref{relatxi} by $\text{sign}(\xi_{n}(t))$ and integrating over the interval $[ t, T]$, we then obtain
\begin{equation}\label{estxi}
 \frac{1}{\varepsilon_n} \int_{t}^{T}\vert \xi_{n}(s)\vert ds \leq \bar{k} \quad \text{for all }t\in [0,T]
\end{equation}
for a suitable constant $\bar{k}$ independent of $n$.

We are now ready to obtain the $L_{1}$ uniform boundedness of the sequence $\lbrace \dot{p}_{\varepsilon}\rbrace$.
\begin{lem}\label{pndet}
The sequence  $\lbrace \dot{p}_n(\cdot)\rbrace$ is bounded in $L^{1}([ 0, T] ,\R^{n})$, uniformly with respect to $n$ .
\end{lem}
\begin{proof}
We recall that that the adjoint equation can be rewritten as 
\[
 -\dot{p}_n(t)= -\frac{\delta_n(t)}{\varepsilon_n}\delta_n''(t)p_n(t) -\frac{\xi_n(t)}{\varepsilon_n}\delta_n'(t) + \nabla_x f(x_n(t),u_n(t))p_n(t).
\]
The result follows immediately by using Lemma \ref{lemma1} together with \eqref{estxi}, \eqref{bddn}, and the assumptions on $f$.
\end{proof}
\subsection{Passing to the limit}
We wish now to derive the equations which are satisfied by a suitable limit of the sequence $\lbrace p_{n}\rbrace$. 
By possibly extracting a further subsequence from $\lbrace \varepsilon_{n}\rbrace $ (without relabeling), 
thanks to Lemma \ref{pndet} and Helly's selection theorem, we can suppose that there exists a function $p \in BV([ 0, T] ;\R^{n})$ such that 
\[
p_{n}(t)\rightarrow p(t) \quad \text{for all } t\in [ 0, T]
\]
(in particular $p(T)=-\nabla h(x_{\ast}(T))$) and, for all $h \in {\mathcal C}^{0} ([ 0, T] ;\R^{n})$,
\[
\int_{0}^{T} \langle h(t),\dot{p}_{n}(t)\rangle dt \rightarrow \int_{0}^{T} \langle h(t),dp\rangle.
\]
We recall also that   
\begin{alignat*}{2}
& x_{n}\rightarrow x_{\ast}&\qquad &\text{uniformly in }[ 0, T]\\
& \dot{x}_{n}^{\ast}\rightharpoonup \dot{x}^{\ast}&\qquad &\text{weakly in }L^{2}([ 0, T] ;\R^{n})\\
& u_{n}\rightarrow u_{\ast}&\qquad &\text{strongly in }L^{2}([ 0, T] ;\R^{m})\\
& u_{n}(t) \rightarrow u_{\ast}(t)&\qquad &\text{a.e. on }[ 0, T],
\end{alignat*}
and that we have set 
\[
I_{\partial}=\lbrace  t\in [ 0, T]: x_{\ast}(t)\in \partial C(t)\rbrace.
\]
We define also 
\[
I_{0}:= [ 0, T]\setminus I_{\partial}=\lbrace  t\in [ 0, T]: x_{\ast}(t)\in {\rm int}\, C(t)\rbrace.
\]
Of course one of the two sets $I_{0}$ and $I_{\partial}$ may be empty. 
We will proceed with our arguments, without loss of generality, by assuming that both of them are nonempty.

For every compact interval $[ s, t]\subset I_{0}$, the adjoint equation for $p_{n}$ is 
\begin{equation}\label{adjI0}
 - \dot{p}_{n}(\tau)= \nabla_{x}f(x_{n}(\tau),u_{n}(\tau)) p_{n}(\tau),
\end{equation}
since $d_{C(t)}(\cdot)$ is zero in a neighborhood of $x_{\ast}(t)$ and $x_{n}$ converges to $x_{\ast}$ uniformly. 
By integrating \eqref{adjI0} over $[ s, t]$ and using the absolute continuity of $p_{n}$, we obtain 
\[
p_{n}(s)-p_{n}(t)=\int_{s}^{t}\nabla_{x}f(x_{n}(\tau),u_{n}(\tau)) p_{n}(\tau)\, d\tau.
\]
Since $p_{n}$ converges to $p$ pointwise and is uniformly bounded, by the dominated convergence theorem we obtain
\[
p(s)-p(t)=\int_{s}^{t}\nabla_{x}f(x_{\ast}(\tau),u_{\ast}(\tau)) p(\tau)\, d\tau.
\]
We have therefore proved the following
\begin{prop}\label{4.1}
On $I_{0}$, $p$ is absolutely continuous and satisfies the equation 
\begin{equation}\label{adj0}
- \dot{p}(t)= \nabla_{x}f(x_{\ast}(t),u_{\ast}(t)) p(t), \quad a.e.\  t\in I_{0}.
\end{equation}    
\end{prop}

\medskip

\noindent We will deal now with passing to the limit along \eqref{adjeps} and obtaining necessary conditions on the whole interval $[ 0, T]$. 
The main effort will be put in passing to the limit in $ I_{\partial}$.

For the sake of convenience, we rewrite here the adjoint equation for $p_{n}$, recalling that we have set
$\xi_{n}(t)= \langle \nabla_{x}d_{C(t)}(x_{n}(t)),p_{n}(t) \rangle$. We have
\[
\begin{split}
 - \dot{p}_{n}(t)& = -\frac{1}{\varepsilon_{n}} \nabla_{x}d_{C(t)}(x_{n}(t))\xi_{n}(t)-\frac{d_{C(t)}(x_{n}(t))}{\varepsilon_{n}}
\nabla_{x}^{2} d_{C(t)}(x_{n}(t))p_{n}(t)\\
&\qquad +\nabla_{x} f(x_{n}(t),u_{n}(t)) p_{n}(t)\\
&\qquad  :=\textbf{I}+\textbf{II}+\textbf{III}
\end{split}
\]
We recall that under our assumptions this equation can be seen as and O.D.E. with (possibly) a switch, which occurs at the time $t_n$, 
and we can assume that the sequence $\{t_n\}$ has a limit point $\bar{t}$ (see \eqref{convtn}).

We discuss now passing to the limit for each summand \textbf{I}, \textbf{II}, and \textbf{III}.\\
$\textbf{I}$. Set $n_{\ast}(t)$ to be the unit external normal to $C(t)$ at $x_{\ast}(t)$ for all $t\in I_{\partial}$ and $0$ for all $t\in I_{0}$ .
Observe that on every compact subset $I\subset [ 0, T]$ such that $\bar{t}\notin I$ we can suppose that $\nabla_{x}d_{C(t)}(x_n(t))$ 
converges to $n_{\ast}(t)$ uniformly on $I$.
By the uniform boundedness of $\nabla_{x}d_{C(t)}(\cdot)$ and \eqref{estxi} we can suppose that (up to a subsequence) along 
$\lbrace \varepsilon_{n}\rbrace$ we have that the sequence of measures
\[
\Big\{ \frac{\xi_{n}(.)}{\varepsilon_{n}}\nabla_{x}d_{C(\cdot)}(x_{n}(\cdot))\, dt \Big\}
\]
converges weakly$^\ast$ in the dual of ${\mathcal C}^{0}([ 0, T] ;\R^{n})$ to a finite signed Radon measure on $[ 0, T]$, which can be written as 
\begin{equation}\label{dmy}
 \xi(t) n_{\ast}(t) d\mu ,
\end{equation}
where $\mu$ is a finite signed Radon measure on $[ 0, T]$ and $\xi \in L^{\infty}[ 0, T]$, $\xi(t)\geq 0$ $\mu$-a.e.
Observe that, without loss of generality, we can suppose that $\xi (t)= 0$ on $I_{0}$.\\
$\textbf{II}$. Recalling \eqref{geps},
\[
\frac{d_{C(t)}(x_n(t))}{\varepsilon_{n}}\leq \beta+\gamma\; \text{ for all $t\in [ 0, T]$ and $n\in \N$.}
\]
Recall that $\nabla_{x}^{2}d_{C(t)}(x_{\ast}(t))=0$ if $ t\in I_{0}$, and set $\nabla_{x}^{2}d_{C(t)}(x_{\ast}(t))=
\nabla_{x}^{2}d_{S}(t,x_{\ast}(t))$ if $t\in I_{\partial}$ 
(indeed, the signed distance $d_{S}(t,\cdot)$ is $C^{2}$ in a neighborhood of boundary points of $C(t)$, $t\in I_\partial$,
see, e.g., Proposition 2.2.2 (iii) in \cite{CS}, since both $(M_1)$ and $(M_2)$ imply that $C(t)$ has nonempty interior).\\
We can suppose that $\nabla_{x}^{2}d_{C(t)}(x_{n}(t))$ converges uniformly to 
$\nabla_{x}^{2}d_{C(t)}(x_{\ast}(t))$ on every compact $I\in [ 0, T]$ such that $\bar{t} \notin I$.
By combining the uniform bound on $d_{C(t)}(x_{n}(t))/\varepsilon_{n}$, the uniform convergence of 
$\nabla_{x}^{2}d_{C(t)}(x_{n}(t))$ on every compact $ I \in [ 0, T]$ with $\bar{t}\notin I$ and 
the pointwise convergence of $p_{n}$, we obtain, up to a subsequence without relabeling, that 
\[
\frac{d_{C(t)}(x_{n}(t))}{\varepsilon_{n}}\nabla_{x}^{2}d_{C(t)}(x_{n}(t))p_{n}(t)\rightharpoonup \eta (t)\nabla_{x}^{2}d_{C(t)}(x_{\ast}(t))p(t)
\]
weakly in $ L^{2}([ 0, T];\R^{n})$, where $\eta \in L^{\infty}[ 0, T]$, $0\leq \eta (t) \leq \beta + \gamma$ a.e. 
Observe also that $\eta (t)\equiv 0 $ on $I_{0}$.\\
$\textbf{III}$. Recalling Proposition \ref{cowlambda}, up to a subsequence 
\[
\nabla_{x} f(x_{n}(t),u_{n}(t)) p_{n}(t)\rightarrow \nabla_{x} f(x_{\ast}(t),u_{\ast}(t)) p(t)\;\text{ a.e. on }[ 0, T]
\]
and weakly in $ L^{2}([ 0, T];\R^{n})$.

\bigskip

We have therefore proved that $p$ satisfies in a weak sense a differential equation, namely (and this establishes \eqref{adj00}) we have
\begin{prop}\label{eqp} Let $p$ be a weak limit of $p_{n}$ in $BV ([ 0, T];\R^{n})$. Then $p(T)= -\nabla h(x_{\ast}(T))$ and there exist a 
finite Radon measure $\mu$ on $[ 0, T]$, and nonnegative measurable functions $\xi ,\eta : [ 0, T]\rightarrow \R$ 
satisfying the properties $\xi\in L^1$, $\xi (t) = 0$ on $I_{0}$ and $\xi(t) \ge 0$ on $I_\partial$, $\mu$-a.e., and
$0\leq \eta (t) \leq \beta + \gamma$, $\eta(t)= 0$ on $I_{0}$, a.e., such that 
for all continuous functions $\varphi : [ 0, T]\rightarrow \R^{n}$ we have
\begin{equation}\label{adj}
\begin{split}
-\int_{[ 0, T]}\langle \varphi(t),dp(t)\rangle & +\int_{[ 0, T]}\langle \varphi(t),n_{\ast}(t)\rangle\xi(t)\, d\mu 
-\int_{[ 0, T]}\langle \varphi(t),\nabla_{x} f(x_{\ast}(t),u_{\ast}(t)) p(t)\rangle\, dt 
\\
&= -\int_{[ 0, T]}\langle \varphi(t),\eta(t) \nabla_{x}^{2}d_{C(t)}(x_{\ast}(t))p(t) \rangle\, dt,
\end{split}
\end{equation} 
where we recall that $n_\ast(t) = 0$ if $x_\ast (t)\in \text{int}\, C(t)$, and $n_\ast (t) = \nabla_x d_S(t,x_\ast (t))$ is the unit external
normal to $C(t)$ if $x_\ast (t)\in \partial C(t)$.
\end{prop}
The adjoint vector $p$ can be proved to satisfy a bunch of further conditions in the interval $I_{\partial}$.
\begin{prop}\label{condp}
Let $p, \xi, \eta $ be given by Proposition \eqref{eqp} and set $p^{N}(t)=\langle p(t),n_{\ast}(t)\rangle$ for all  
$t \in [ 0, T]$. Then 
\begin{enumerate}
\item[\textbf{(1)}]
$ p^{N}(t)= 0$ for a.e. $ t \in [ 0, T]$, and $p$ is absolutely continuous on $I_0$.
\item[\textbf{(2)}] If \eqref{m1} holds, then (recall that $I_{\partial}= [ \bar{t}, T]$ according to Proposition \ref{Im1} and \eqref{convtn})
$p$ is absolutely continous on $(\bar{t},T)$ and for a.e. $t\in [\bar{t}, T]$ we  have 
\begin{equation}\label{adjId}
- \dot{p}(t)= \langle \dot{n}_{\ast}(t),p(t)\rangle n_{\ast}(t)+ \Gamma(t) p(t) -\langle \Gamma(t) p(t),n_{\ast}(t)\rangle n_{\ast}(t),
\end{equation}
where $\Gamma(t)= \nabla_{x} f(x_{\ast}(t),u_{\ast}(t))-\eta(t)\nabla_{x}^{2}d_{C(t)}(x_{\ast}(t))$. Moreover, the equalities
\begin{align}\label{ptb}
p(\bar{t}-)-p(\bar{t}+)&= p^{N}(\bar{t}-) n_{\ast}(\bar{t})\\
\label{ptc}
p(T-)-p(T)&= - p^{N}(T) n_{\ast}(T)
\end{align}
(which mean that jumps may occur only in the normal direction $n_\ast$) are valid.
\item[\textbf{(3)}] If \eqref{m2} holds, then 
\begin{equation}\label{pt2}
p(0)-p(0+)=\big( p^{N}(0)-p^{N}(0+)\big) n_{\ast}(0).
\end{equation}
\item[\textbf{(4)}] Finally, if \eqref{m1} holds
$p$ is continuous at $\bar{t}$.
\end{enumerate}
\end{prop}
\begin{rem}\label{2}
It follows from the above Proposition that the measure $\mu $ appearing in \eqref{adj} may admit a Dirac 
mass at most at $t=0$ (if \eqref{m2} holds) or at $t=T$ (if \eqref{m1} holds).
\end{rem}
\begin{proof}
\textbf{(1)}. The first assertion is an immediate consequence of \eqref{estxi}, 
which implies that the sequence $\langle p_{n}(t),\nabla_{x}d_{C(t)}(x_{n}(t))\rangle $ 
converges to $0$ in $ L^{1}( 0, T)$, and the convergence of $\nabla_{x}d_{C(t)}(x_{n}(t))$ to $n_{\ast}(t)$ for all 
$ t \in [ 0, T]$, $t\neq \bar{t}$ and of $p_{n}(t)$ to $p(t)$. The second assertion follows from Proposition \ref{4.1}.

\textbf{(2)} and \textbf{(3)}. Since $n_{\ast}(t)$ is continuous on $I_{\partial}$, there exist $n-1$ continuous unit vectors 
$v_{1}(t)$,\ldots , $v_{n-1}(t)$ such that $\R^{n}= \R n_{\ast}(t)\oplus \text{span}\, \langle v_{1}(t),\ldots ,v_{n-1}(t)\rangle $.
Fix $t\in (\bar{t}, T)$ and $\sigma >0$  such that $[ t-\sigma,t+\sigma]\subset (\bar{t}, T)$. 
Let $\varphi : [ 0, T]\rightarrow \R^{n}$ be continuous, with support contained in $[ t-\sigma,t+\sigma]$ .
Set $\varphi^{T}(t)=\varphi(t)-\langle \varphi(t),n_{\ast}(t)\rangle n_{\ast}(t)$.
By putting $\varphi^{T}(t)$ in place of $\varphi$ in \eqref{adj} we obtain 
\begin{equation}  \label{intp}
\begin{split}
-\int_{t-\sigma}^{t+\sigma}\langle \varphi^{T}(s),dp(s)\rangle &+\int_{t-\sigma}^{t+\sigma}\langle \varphi^{T}(s),n_{\ast}(s)\rangle\xi(s)\, d\mu
= \int_{t-\sigma}^{t+\sigma}\langle \varphi^{T}(s),\nabla_{x} f(x_{\ast}(s),u_{\ast}(s)) p(s)\rangle \, ds 
\\
 &\qquad\qquad\quad  -\int_{t-\sigma}^{t+\sigma}\langle \varphi^{T}(s),\eta(s) \nabla_{x}^{2}d_{C(s)}(x_{\ast}(s))p(s) \rangle\, ds.
\end{split}
\end{equation} 
Observe now that $\langle \varphi^{T},n_{\ast}(t)\rangle \equiv 0$, so that, by letting $\sigma \rightarrow 0$ 
in the above equation and using the continuity of $\varphi^{T}$, we obtain
$\langle \varphi^{T}(t),p(t+)-p(t-)\rangle =0$,
namely $\langle p(t+)-p(t-),\varphi(t)\rangle =\langle p(t+)-p(t-),\langle \varphi(t),n_{\ast}(t)\rangle n_{\ast}(t)\rangle $.
By taking subsequently $\varphi$ such that $\varphi (t)=n_{\ast}(t)$, and $\varphi(t)=v_{i}(t)$, we obtain that 
$ p(t-)-p(t+)=( p^{N}(t-)-p^{N}(t+)) n_{\ast}(t)$, namely jumps of $p$ may occur only in the direction $n_{\ast}(t)$, for all $ t \in ( \bar{t}, T)$.
By taking $t=T$ and arguing as in \eqref{intp}, but integrating over $ [ T-\sigma, T]$, one immediately obtains \eqref{ptc}.
In order to prove \eqref{ptb}, resp. \eqref{pt2}, it is enough to extend $n_{\ast}(t)$ to be constantly 
$n_{\ast}(\bar{t})$ on $[ \bar{t}-\sigma, \bar{t})$, resp. constantly $n_{\ast}(0)$ on $(0,\sigma)$, and observe that the part (1) of this Proposition
together with the fact that $p$ has bounded variation imply that $p(\bar{t}+)=p(T-)=0$.

Fix now an interval $[s,t]\subseteq (\bar{t},T)$. The regularity condition on $\partial C(t)$ allows us 
to integrate by parts on $(s,t)$ (see (34), p. 8 in \cite{MMM}), so that
\begin{equation}\label{intpart}
 \int_{s}^{t} \langle n_{\ast}(\tau),dp(\tau)\rangle +\int_{s}^{t}\langle \dot{n}_{\ast}(\tau),p(\tau)\rangle\,  
d\tau =\langle n_{\ast}(t+),p(t+)\rangle - \langle n_{\ast}(s-),p(s-)\rangle =0 ,
 \end{equation}
where both summands in the right hand side of \eqref{intpart} vanish,  
as a consequence of \textbf{(1)} and of the fact that $p$ has bounded variation, since both $s$ and $t$ belong to the interior of $I_\partial$.
In other words, the two measures $\langle n_\ast , dp\rangle$ and $\langle \dot{n}_\ast, p\rangle\, dt$ coincide in the open
interval $(\bar{t},T)$.
Therefore, for all continuous $\varphi$, with support contained in  $(\bar{t}, T)$, 
we obtain from \eqref{intp} and \eqref{intpart} that 
\[
\begin{split}
 - \int_{\bar{t}}^{T} \langle \varphi(t),dp (t)\rangle &=- \int_{\bar{t}}^T \langle \varphi (t),n_\ast(t)\rangle\, \langle n_\ast(t),dp(t)\rangle - 
 \int_{\bar{t}}^T \langle \varphi^T (t),dp (t)\rangle 
 \\
 & = \int_{\bar{t}}^{T} \langle \varphi(t), n_{\ast}(t)\rangle  \langle p(t),\dot{n}_{\ast}(t)\rangle \, dt\\
&\qquad +
  \int_{\bar{t}}^{T} \langle \varphi(t)-\langle \varphi(t),n_{\ast}(t)\rangle n_{\ast}(t),\nabla_{x} 
  f(x_{\ast}(t),u_{\ast}(t)) p(t)\rangle dt 
\\
&\qquad - \int_{\bar{t}}^{T} \langle \varphi(t)-\langle \varphi(t),n_{\ast}(t)\rangle n_{\ast}(t),\eta(t) \nabla_{x}^{2}d_{C(t)}(x_{\ast}(t))p(t) \rangle dt
 \end{split}
\]
Since $\varphi$ is arbitrary, we obtain \eqref{adjId}.

\textbf{(4)}. By multiplying \eqref{adjeps} by $\frac{p_n(s)}{\vert p_n(s)\vert}$ and integrating over 
$[ \bar{t}-\sigma, \bar{t}+\sigma]$, using the fact that $d_{C(t)}(\cdot)$ is $C^{1,1}$, uniformly with respect to $t$, 
by using the same argument of the proof of Lemma \ref{lemma1} (see \eqref{rrrprox}) we obtain
\[
\vert p_n(\bar{t}-\sigma)\vert - \vert p_n (\bar{t}+\sigma)\vert \leq 
k \int_{\bar{t}-\sigma}^{\bar{t}+\sigma}\vert p_n(s)\vert ds \leq k \sigma
\]
for a suitable constant $k$ independent of $n$.
By passing to the limit as $n \rightarrow \infty$ (along a suitable subsequence) we obtain
\[
\vert p(\bar{t}-\sigma)\vert - \vert p (\bar{t}+\sigma)\vert \leq 0 ,
\]
whence
\begin{center}
$\vert p(\bar{t}-) \vert \leq \vert p(\bar{t}+)\vert $.
\end{center}
Analogously, multiplying both sides of \eqref{relatxi} by sign $(\xi_n(t))$, integrating and using the fact 
that $\frac{\vert \xi_n\vert }{\varepsilon_n}$ is uniformly bounded in $L^{1}(0,T)$, we obtain by passing to the 
limit as $n \rightarrow \infty$ that 
$\vert p^{N}(\bar{t}-) \vert \leq \vert p^{N}(\bar{t}+)\vert = 0$. The latter vanishes,
recalling \textbf{(1)}, and thus it follows that $p^{N}(\bar{t}-) = 0 $ as well. Recalling \eqref{ptb}, $p$ is continous at $\bar{t}$.
\end{proof}
Our last task is now to the limit formulation of the maximum principle. 
Indeed, from \eqref{PMPeps} we immediately obtain, by passing to the limit for $n \rightarrow \infty $, that \eqref{PMP} holds.

\medskip

Therefore, the proof of our main result is concluded.
\section{An example}\label{Ex}
We propose a simple example, inspired by Remark 5.1 in \cite{SSER}, in order to test our necessary conditions. 
\subsection{Example 1}
The state space is $\R^{2}\ni (x,y)$, the constraint $C(t)$ is constant and equals $C:= \lbrace (x,y) : y\geq 0\rbrace$, the upper half plane.

We wish to minimize $x(1)+y(1)$ subject to
\begin{equation}\label{(E1)}
\begin{cases}
\big(\dot{x}(t),\dot{y}(t)\big)&\in \,  -N_{C}\big(x(t),y(t)\big) + \big(u^{x}(t),u^{y}(t)\big)
\\
\big(x(0),y(0)\big)& =\big(0,y_{0}\big),\qquad y_{0}\geq 0,
\end{cases}
\end{equation}
where the controls $\big(u^{x}(t),u^{y}(t)\big)$ belong to $[ -1, 1]\times [ -1, -1/2]=:U$.

By inspecting the level sets of the cost $h\big((x,y)\big)=x+y$, it is natural looking for an optimal solution such 
that both $u^x$ and $u^y$ are nonpositive.
If we restrict ourselves to the case where $u^y(t)<0$ for a.e. $t$, then this problem satisfies all our assumptions; 
in particular we are in the case \eqref{m1}.

Observe first that if $y_{0}\geq 1$, the constraint $C$ does not play any role, and the optimality of the control $(-1,-1)$ 
is straightforward. If instead $0\leq y_{0}< 1$, then our analysis becomes relevant.
Since we are in the case \eqref{m1}, there exists at most one $\bar{t}$ such that the optimal solution hits the boundary of 
$C$ and after $\bar{t}$ it remains on $\partial C$. The external unit normal $n_{\ast}(t)$ is identically $(0,-1)$ and on 
$\partial C$, namely for $x=0$, we have $\nabla_{x}^{2}d_C\big((0,y)\big)\equiv 0$.
Thanks to Propositions \eqref{condp} and \eqref{adj0} we obtain, for the optimal trajectory $(x_{\ast},y_{\ast})$ 
corresponding to the optimal control $(u_{\ast}^{x},u_{\ast}^{y})$ and the adjoint vector $(p^{x},p^{y})$, that $(p^x,p^y)$
is absolutely continuous on $(0,1)$,
$\dot{p}^{x}=0$, $\dot{p}^{y}=0$ a.e. on $[ 0, T]$, $p^{x}(1)=p^{y}(1)=-1$,
$p^{x}$ is continuous at $t=1$ and $p^{y}(1-)+1=1,$ namely $p^{y}(1-)=0$.
Thus the adjoint vector $(p^{x},p^{y})$ is :
\begin{alignat*}{2}
p^{x}(t)&=-1& \quad&\text{for all }t\in [ 0, 1]\\
p^{y}(t)&=0 &\quad&\text{for all  }t\in [ 0, 1)\\
p^{y}(1)&=-1&&\\
\mu &=-\delta_1.
\end{alignat*}
The maximum condition reads as 
\[
\begin{split}
   \langle (-1,-1),(u_{\ast}^{x},u_{\ast}^{y})\rangle &= \max_{\vert u_{1}\vert\leq 1 , -1\leq u_{2}\leq -1/2}\langle (-1,-1),(u_{1},u_{2})\rangle 
\quad \text{for} \ t=1\\
   \langle (-1,0),(u_{\ast}^{x},u_{\ast}^{y})\rangle &= \max_{\vert u_{1}\vert\leq 1 , -1\leq u_{2}\leq -1/2}\langle (-1,0),(u_{1},u_{2})\rangle 
\quad \text{for} \ 0\leq t<1,
\end{split}
\]
which gives $u_{\ast}^{x}=-1$, while no information is available for $u_{\ast}^{y}(t)$.
If we assume that  $u_{\ast}^{y}$ is continuous at $t=1$, then the transversality condition yields  and $u_{\ast}^{y}(1)=-1$. If we
assume that $u_\ast^y$ is constant,
then an expected optimal control $u_{\ast}^{y}= -1$ is found. Of course all other optimal
controls $u^y_\ast $, namely $u^y_\ast (t) =-1$ for $0\le t <\bar{t}$ and $u^y_\ast (t) \le 0$ for $ \bar{t} < t < 1$ satisfy our necessary conditions.
\begin{rem}
1) The vanishing of $p^y$ on the interval $[\bar{t},1]$ is somehow to be expected, since all controls $u^y\le -1/2$ (actually $u^y\le 0$)
in that time interval are optimal.
The vanishing of $p^y$ on $[0,\bar{t}]$ instead makes a remarkable difference with the classical case 
(i.e., $C=\R^n$), where $p^y\equiv -1$ allows
to fully determine the optimal control. It should be natural finding an adjoint vector which gives the same information as in the classical
case in an interval where the optimal solution lies in the interior of $C$, but this feature does not follow from the method developed here.\\
2) In order to have the assumption $(M_1)$ be satisfied, we had to impose that the control $u^y$ was negative and 
bounded away from zero. However,
all arguments of Section \ref{convergence} go through also in the case where $u^y$ belongs to the interval $[-1,1]$. In fact, the optimal control
$u_n^y$ for the approximate problem is always $-1$, so that the approximate solution $(x_n,y_n)$ 
touches the boundary of $C$ only at one time.
\end{rem}
\section{Conclusions}
Given a smooth moving set $C(\cdot)$ and smooth maps $f$ and $h$, we have proved necessary optimality conditions for global minimizers
of the problem $(P)$, provided the strong inward/outward pointing conditions $(M_1)$ or $(M_2)$ are satisfied. 
Such conditions were imposed in order
to deal with the discontinuity of the gradient of the distance to $C(t)$ at boundary points and actually transform the space discontinuity
into a time discontinuity. A similar idea appears in \cite{BBT}. 
An alternative approach is adopting the method developed
in \cite{BPK}, which makes use of a smooth approximation of the distance. This approach, however, requires the uniform strict convexity of
the moving set. 

If $C(t)$ is a moving smooth manifold without boundary, in particular has empty interior, then the squared distance $d_{C(t)}^2(\cdot)$
is of class $\mathcal{C}^2$ in a whole neighborhood of $C(t)$. In this case, then, all arguments of Section \ref{convergence} go through
as well. In particular, the proof of Lemma \ref{lemma1} does not require $(M_1)$ or $(M_2)$. Our main results can be rephrased in this
context, but for the sake of brevity we do not deal with such details.

\end{document}